\documentclass[11pt, reqno, english]{amsart}
\usepackage{accents}
\usepackage{style}
\usepackage[normalem]{ulem}
\usepackage{wasysym}

\renewcommand{\conv}{\operatorname{conv}}

\newcommand{\avg}{\operatorname{avg}}
\newcommand{\vor}{\operatorname{Vor}}

\begin{document}

\title{
An optimal Brouwer's fixed point theorem for discontinuous functions
}

\author[Adams]{Henry Adams}
\address[HA]{Department of Mathematics, University of Florida, Gainesville, FL 32611, USA}
\email{henry.adams@ufl.edu}

\author[Frick]{Florian Frick}
\address[FF]{Dept.\ Math.\ Sciences, Carnegie Mellon University, Pittsburgh, PA 15213, USA}
\email{frick@cmu.edu} 

\thanks{\hspace{-4mm}\textit{2020 Mathematics Subject Classification.} 54H25, 52A40. \\
% 55N31, 55U10
HA was supported by the Simons Foundation Travel Support for Mathematicians.
FF was supported by NSF CAREER Grant DMS 2042428.}

\begin{abstract}
Brouwer's fixed point theorem states that any continuous function from a closed $n$-dimensional ball to itself has a fixed point.
In 1961, Klee showed that if such a function has discontinuities that are bounded, then it has a point that is close to being fixed.
We improve upon Klee's results in any finite-dimensional Euclidean space, and prove that our bounds are the best possible.
\end{abstract}

\maketitle

%\setcounter{tocdepth}{1}
%\tableofcontents

\section{Introduction}

Brouwer's fixed point theorem~\cite{brouwer1911abbildung} says that if $f$ is a continuous function from the closed unit ball to itself, then $f$ has a fixed point, i.e.\ a point $x$ with $f(x)=x$.
This theorem was used to prove the existence of Nash equilibria~\cite{nash1950equilibrium}, and~\cite{bjorner2017using} surveys a number of inviting applications and extensions of Brouwer's fixed point theorem.
What can be said if $f$ is not continuous?

Let $X$ be a metric space.
A function $f\colon X \to X$ is \emph{$\varepsilon$-continuous} if each point $x\in X$ admits a neighborhood $U_x$ such that the image $f(U_x)$ has diameter at most $\varepsilon$.
A point $x\in X$ is an \emph{$\varepsilon$-fixed point for $f$} if $\|x-f(x)\| \le \varepsilon$.
In~\cite[Theorem~3]{klee1961stability}, Klee shows that if $X$ is a convex compact subset of a normed linear space, and if $f\colon X \to X$ is $\varepsilon$-continuous, then for any $\varepsilon'>\varepsilon$ the function $f$ has an $\varepsilon'$-fixed point.

Closely related is the \emph{modulus of discontinuity}~$\delta(f)$ of~$f\colon X\to X$ defined as
\[
    \delta(f) = \inf\{\delta \ge 0 \ | \ \forall x \in X,\ \exists \ \text{neighborhood} \ U_x \ \text{of} \ x \ \text{such that} \ \diam(f(U_x)) \le \delta\};
\]
see for example~\cite{dubins1981equidiscontinuity} and the recent paper~\cite{adams2025quantifying}.
Observe that if $f$ is $\varepsilon$-continuous then $\delta(f) \le \varepsilon$, and conversely if $\delta(f) \le \varepsilon$ then $f$ is $\varepsilon'$-continuous for all $\varepsilon' > \varepsilon$.

We show how to improve upon Klee's result in any finite-dimensional Euclidean space, where moreover we prove that our bounds are optimal.
For $n\ge 1$, let $B^n=\{x \in \R^n~|~\|x\|\le 1\}$ be the unit ball in $n$-dimensional Euclidean space.
Let $R_n = \sqrt{2(n+1)/n}$ be the diameter of the inscribed $n$-simplex in $B^n$, measured using the Euclidean metric.

\begin{theorem-main}
If $f\colon B^n \to B^n$ is $\varepsilon$-continuous, then for any $\varepsilon' > \frac{\varepsilon}{R_n}$ the function $f$ has an $\varepsilon'$-fixed point.
\end{theorem-main}

Hence in $n$-dimensional Euclidean space $\R^n$, the Main Theorem improves upon~\cite[Theorem~3]{klee1961stability} by a factor $R_n$.
Note $R_1=2$, $R_2=\sqrt{3}$, $R_3=\sqrt{8/3}$, and $R_n$ decreases monotonically towards $\sqrt{2}$ as $n$ increases towards $\infty$.

The following result shows that the bound in our Main Theorem is optimal.
We can restrict attention to $0 < \varepsilon \le 2$ since any function $f\colon B^n \to B^n$ is $2$-continuous.

\begin{optimality-result}
Let $0 < \varepsilon \le 2$.
For any $\varepsilon' < \frac{\varepsilon}{R_n}$, there exists an $\varepsilon$-continuous function $f\colon B^n \to B^n$ that has no $\varepsilon'$-fixed point.
\end{optimality-result}

In Section~\ref{sec:dim1}, we give simple proofs of the Main Theorem and Optimality Result in $1$-dimensional space, i.e.\ when $n=1$.
We prove the Main Theorem for all dimensions in Section~\ref{sec:main}.
In Section~\ref{sec:optimality} we show the Optimality Result.

\section{The $1$-dimensional case}
\label{sec:dim1}

Note $B_1=[-1,1]$ and $R_1=2$.
We view $[-1,1]$ as a horizontal interval, increasing from left to right.
To show the Main Theorem in the particular case $n=1$, we must show that if $f\colon [-1,1] \to [-1,1]$ is $\varepsilon$-continuous, then for any $\varepsilon' > \frac{\varepsilon}{2}$ the function $f$ has an $\varepsilon'$-fixed point.
We prove the contrapositive.
Let $\varepsilon' > \frac{\varepsilon}{2}$.
We want to show that if no point $x$ satisfies $|x-f(x)|\le \varepsilon'$, then $f$ is not $\varepsilon$-continuous.
Since $-1$ is not an $\varepsilon'$-fixed point, $f$ must map $-1$ to some point further to the right, i.e.\ $f(-1)-(-1) > \varepsilon'$.
Similarly, since $1$ is not an $\varepsilon'$-fixed point, $f$ must map $1$ to some point further to the left, i.e.\ $1-(f(1)) < \varepsilon'$.
Decompose the interval $[-1,1]$ into two sets, $[-1,1]=S_R \coprod S_L$, where $S_R$ contains the points that $f$ maps to the right by more than $\varepsilon'$, and where $S_L$ contains the points that $f$ maps to the left by more than $\varepsilon'$:
\begin{align*}
S_R&=\{x \in [-1,1]~|~f(x)-x > \varepsilon'\} \\
S_L&=\{x \in [-1,1]~|~x-f(x) > \varepsilon'\}
\end{align*}
The sets $S_R \ni -1$ and $S_L \ni 1$ are nonempty and cover the connected interval $[-1,1]$, so their closures intersect.
Let $y$ be a point in this intersection of the closures.
Any open neighborhood $U_y$ about $y$ intersects both a point $y_R\in S_R$ arbitrarily close to $y$ and a point $y_L\in S_L$ arbitrarily close to $y$; for our purposes it suffices to assume that $|y-y_R|,|y-y_L|<\varepsilon'-\frac{\varepsilon}{2}$, so that $|y_L-y_R|<2\varepsilon'-\varepsilon$.
Therefore, $f(U_y)$ has diameter at least 
\begin{align*}
f(y_R)-f(y_L)
&= \left(f(y_R) - y_R\right) + \left(y_L - f(y_L)\right) - (y_L-y_R) \\
&> \varepsilon' + \varepsilon' - |y_L-y_R| \\
&> 2\varepsilon' - (2\varepsilon'-\varepsilon) \\
&= \varepsilon.
\end{align*}
So $f$ is not $\varepsilon$-continuous.

We now show the Optimality Result in the particular case $n=1$.
Let $0<\varepsilon\le 2$, and let $\varepsilon' < \frac{\varepsilon}{2}$.
Define the function $f\colon [-1,1] \to [-1,1]$ by
\[ f(x) = \begin{cases}
\tfrac{\varepsilon}{2} & \text{if } x \le 0 \\
-\tfrac{\varepsilon}{2} & \text{if } x > 0.
\end{cases} \]
To see that the function $f$ is $\varepsilon$-continuous, note that the diameter of the image of $f$ is $\varepsilon$.
The fact that the function $f$ has no $\varepsilon'$-fixed points follows since $\varepsilon' < \frac{\varepsilon}{2}$.

\section{Proof of the Main Theorem}
\label{sec:main}

Our proof of the Main Theorem for general dimensions $n$ will rely on two lemmas.
The first lemma, which is essentially Jung's Theorem, bounds the distance from a set of points in $\R^n$ to its convex hull.
For a subset $X\subseteq \R^n$ let $\conv(X)$ denote its convex hull.

\begin{lemma}[Jung's Theorem~\cite{Jung1910}]
\label{lem:Jung}
Let $X$ be a finite set of points in $\R^n$ with diameter at most $\varepsilon$.
Then any point in the convex hull $\conv(X)$ is within distance $\frac{\varepsilon}{R_n}$ from some point in $X$.
\end{lemma}

\begin{proof}
We may assume that $C\coloneqq \conv(X)$ is full-dimensional, by the monotonicity of the $R_n$'s.
For $y\in C$, define $B_C[y;r]=\{z\in C~|~d(z,y)\le r\}$.
Since $C \subseteq \R^n$ is a convex set of diameter at most~$\varepsilon$, Jung's Theorem~\cite{Jung1910} states that there is some point $y\in C$ satisfying $C \subseteq B_C[y;\frac{\varepsilon}{R_n}]$.
Thus all sets in the collection $\{B_C[x;\frac{\varepsilon}{R_n}]\}_{x\in X}$ have a common point, which implies that the nerve of this collection is contractible.
Hence $\bigcup_{x\in X} B_C[x; \frac{\varepsilon}{R_n}]$ is contractible by the nerve lemma.
By induction on dimension the boundary of $C$ is contained in $\bigcup_{x\in X} B_C[x; \frac{\varepsilon}{R_n}]$.
A contractible subset of a convex set that contains the boundary must be the entire convex set.
Thus $C \subseteq \bigcup_{x\in X} B_C[x; \frac{\varepsilon}{R_n}]$.
%It follows that $y \in \cap_{x\in X} B[x;\frac{\varepsilon}{R_n}]$. Furthermore, if $\sigma \subsetneq X$ is a proper face of $\conv(X)$, then
%\[\conv\left(\sigma\cup\{y\}\right) \subseteq \bigcup_{x\in\sigma}B\left[x;\tfrac{\varepsilon}{R_n}\right].\]
%As we vary over all proper faces $\sigma$ of $\conv(X)$, the sets $\conv(\sigma\cup\{y\})$ cover $\conv(X)$.
%Therefore, the balls $\{B[x;\tfrac{\varepsilon}{R_n}]\}_{x\in X}$ also cover $\conv(X)$, as desired.
\end{proof}

The second lemma we use involves Vietoris--Rips simplicial complexes~\cite{Vietoris27}.
For $X$ a metric space and for $r\ge 0$, the \emph{Vietoris--Rips simplicial complex} $\vr{X}{r}$ has $X$ as its vertex set, and contains a finite subset $\sigma\subseteq X$ as a simplex if $\diam(\sigma)\le r$.

This second lemma is the same as~\cite[Lemma~7.4]{GH-BU-VR}, except with two minor changes.
First, we replace the modulus of discontinuity with $\varepsilon$-continuity, which we are allowed to do because if a function is $\varepsilon$-continuous, then its modulus of discontinuity is at most~$\varepsilon$.
Second, while in~\cite{GH-BU-VR} the existence of a parameter $\alpha$ is asserted, the proof there actually shows that the conclusion holds for every sufficiently small~$\alpha > 0$; this is the generality that we will require here.

\begin{lemma}
\label{lem:ind-map}
Let $f\colon X\to Y$ be an $\varepsilon$-continuous function between metric spaces, with $X$ compact.
Then for every $\gamma>0$ and every sufficiently small $\alpha>0$ the function~$f$ induces a (continuous) simplicial map $\overline{f}\colon \vr{X}{\alpha}\to \vr{Y}{\varepsilon + \gamma}$ defined by $\overline{f}(\{x_0,...,x_m\}) = \{f(x_0),...,f(x_m)\}$.
\end{lemma}

We now prove the Main Theorem.

\begin{proof}[Proof of the Main Theorem]
Let $f\colon B^n \to B^n$ be $\varepsilon$-continuous.
We must show that for any $\varepsilon' > \frac{\varepsilon}{R_n}$ the function $f$ has an $\varepsilon'$-fixed point.
Let $0<\gamma<R_n\varepsilon'-\varepsilon$.
Let $\alpha>0$ satisfy Lemma~\ref{lem:ind-map} for $f\colon B^n \to B^n$.
If necessary, further decrease $\alpha>0$ to be sufficiently small so that $\frac{\varepsilon+\gamma}{R_n} + \frac{\alpha}{2} < \varepsilon'$.
Pick a subset $Z \subseteq B^n$ that is dense enough so that the balls of radius $\frac{\alpha}{2}$ centered at the points in $Z$ cover $B^n$.
Use a partition of unity on these balls to get a continuous map $\iota \colon B^n \to \vr{Z}{\alpha}$.
The $\varepsilon$-continuous restriction function $f|_Z \colon Z \to B^n$ induces a continuous map $\overline{f} \colon \vr{Z}{\alpha} \to \vr{B^n}{\varepsilon+\gamma}$ with $\overline{f}(\{z_0,...,z_m\}) = \{f(z_0),...,f(z_m)\}$, as in Lemma~\ref{lem:ind-map}.
Finally, define a continuous map $\avg \colon \vr{B^n}{\varepsilon+\gamma} \to B^n$ as follows.
A point in the geometric realization of $\vr{B^n}{\varepsilon+\gamma}$ can be written as a formal sum $\sum_{i=0}^m \lambda_i x_i$ with $\lambda_i>0$, with $\sum_i \lambda_i = 1$, with $x_i\in B^n$, and with $\diam(\{x_0,\ldots,x_m\}) \le \varepsilon+\gamma$.
Reinterpret the formal sum $\sum_{i=0}^m \lambda_i x_i$ in the geometric realization as a weighted average of vectors in Euclidean space $\sum_{i=0}^m \lambda_i x_i \in B^n$.
It follows from Lemma~\ref{lem:Jung} (Jung's Theorem) that if $\avg(\sum_{i=0}^m \lambda_i x_i)=x^*$, then some $x_i$ %for $0\le i\le m$
satisfies $\|x_i - x^*\| \le \frac{\varepsilon+\gamma}{R_n}$.

The composite map $F \colon B^n \to B^n$ defined as $F\coloneqq \avg \circ \overline{f} \circ \iota$ is continuous.
\[\begin{tikzcd}
B^n \arrow[r, "\iota"] \arrow[rrr, bend right, "F"] & \vr{Z}{\alpha}  \arrow[r, "\overline{f}"] & \vr{B^n}{\varepsilon+\gamma} \arrow[r, "\avg"] & B^n
\end{tikzcd}\]
Hence by Brouwer's fixed point theorem~\cite{brouwer1911abbildung} applied to the continuous function $F$, there is a fixed point $y \in B^n$ with $F(y)=y$.
Let $\{z_0,\ldots,z_m\}$ be the simplex in $\vr{Z}{\alpha}$ whose convex hull contains $\iota(y)$, which implies each $z_i$ is within distance $\frac{\alpha}{2}$ from $y$.
The image $\overline{f}(\{z_0,\ldots,z_m\})=\{f(z_0),...,f(z_m)\}$ is a simplex in $\vr{B^n}{\varepsilon+\gamma}$ containing $\overline{f}(\iota(y))$ in its convex hull.
Lemma~\ref{lem:Jung} gives $\|f(z_i) - \avg(\overline{f}(\iota(y)))\| \le \frac{\varepsilon+\gamma}{R_n}$ for some $0 \le i \le m$.
But $\avg(\overline{f}(\iota(y))) = F(y) = y$, so $\|f(z_i)-y\| \le \frac{\varepsilon+\gamma}{R_n}$.
And the point $y$ is within $\frac{\alpha}{2}$ of $z_i$, so by the triangle inequality $\|f(z_i)-z_i\| \le \frac{\varepsilon+\gamma}{R_n} + \frac{\alpha}{2} < \varepsilon'$.
Hence $z_i$ is an $\varepsilon'$-fixed point of $f$.
\end{proof}

\section{Proof of the Optimality Result}
\label{sec:optimality}

We show that the bound in our Main Theorem is optimal.

\begin{figure}[htb]
\centering
\includegraphics[width=2.5in]{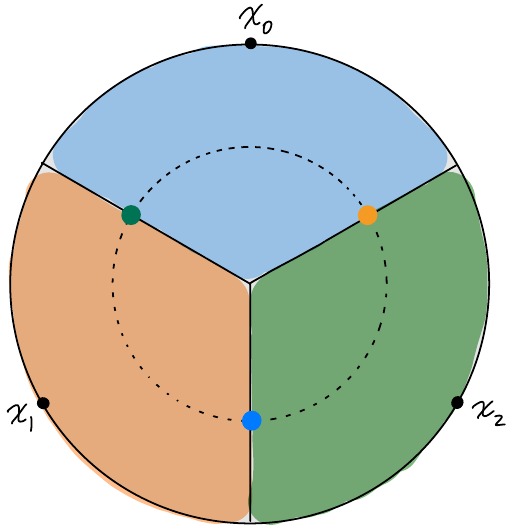}
\caption{
An illustration of the proof of the Optimality Result for $n=2$ (hence $R_2=\sqrt{3}$) and $\varepsilon=1$.
The outer circle is the boundary of $B^2$ and the dotted inner circle has radius $\frac{\varepsilon}{R_2}=\frac{1}{\sqrt{3}}$.
The interiors of the Voronoi regions $\vor(x_i)$ are drawn in blue, orange, and green for $i=0,1,2$, and the same color scheme is used for their corresponding image points $-\frac{\varepsilon}{R_n}\ x_i$ under $f$.
See~\cite[Theorem~5.4]{AAF},~\cite{308972},~\cite[Figure~7]{lim2023gromov},~\cite[Section~2.5]{balitskiy2021bounds},~\cite[Remark~6.10]{akopyan2012borsuk}, and~\cite{sitnikov1958rundheit} for related constructions.}
\end{figure}

\begin{proof}[Proof of Optimality Result]
Let $0 < \varepsilon \le 2$ and let $\varepsilon' < \frac{\varepsilon}{R_n}$.
We must construct an $\varepsilon$-continuous function $f\colon B^n \to B^n$ that has no $\varepsilon'$-fixed point.

Let $\{x_0,\ldots,x_n\}\subseteq B^n$ be the vertex set of a regular $n$-simplex inscribed in the sphere $S^{n-1} \subseteq B^n$.
Note $\diam(\{x_0,\ldots,x_{n+1}\})=R_n$.
Let
\[\vor(x_i) \coloneqq \left\{y \in B^n~:~\|y-x_i\| \le \|y-x_j\| \text{ for all }0\le j\le n \right\}\]
be the Voronoi region about point $x_i$.
The Voronoi regions $\{\vor(x_i)\}_{0\le i\le n}$ form a closed cover of $B^n$.
Therefore, we can choose subsets $V_i \subseteq \vor(x_i)$ for $0\le i\le n$ that form a \emph{partition} of $B^n$, i.e., a disjoint cover of $B^n$.
Define the function $f\colon B^n \to B^n$ by
\[ f(x) = -\frac{\varepsilon}{R_n}\ x_i \quad\quad \text{if } x \in V_i. \]
To see that the function $f$ is $\varepsilon$-continuous, note that the diameter of the image of $f$ is $\varepsilon$.
The fact that the function $f$ has no $\varepsilon'$-fixed points follows since $\varepsilon' < \frac{\varepsilon}{R_n}$.
\end{proof}

%\section{Conclusion and open questions}

%We conclude with some open questions and possible directions for future research.

%\begin{question}
%The paper~\cite{bjorner2017using} surveys a number of inviting applications and extensions of Brouwer's fixed point theorem.
%Can these applications or extensions to be adapted to work for discontinuous functions?
%\end{question}

%\begin{question}
%There are several generalizations of Brouwer's fixed point theorem for infinite-dimensional spaces, including the fixed point theorems of Schauder~\cite{schauder1930fixpunktsatz}, Tikhonov~\cite{tychonoff1935fixpunktsatz}, and Browder~\cite{browder1965nonexpansive}.
%% https://en.wikipedia.org/wiki/Fixed-point_theorems_in_infinite-dimensional_spaces
%Do discontinuous Brouwer's fixed point theorems extend to these settings?
%\end{question}

%\begin{question}
%The Kakutani fixed-point theorem~\cite{kakutani1941generalization} is a generalization of Brouwer's fixed point theorem that holds for set-valued functions (and that was used to prove the existence of Nash equilibria~\cite{nash1950equilibrium}).
%%(\footnote{\url{https://en.wikipedia.org/wiki/Kakutani_fixed-point_theorem}}).
%How do discontinuous Brouwer's fixed point theorems relate?
%\end{question}

%\begin{question}
%What about the case $\varepsilon'=\frac{\varepsilon}{R_n}$?
%I.e., if $f\colon B^n \to B^n$ is $\varepsilon$-continuous, then does $f$ have am $\frac{\varepsilon}{R_n}$-fixed point?
%Or, instead, does there exist an $\varepsilon$-continuous function $f\colon B^n \to B^n$ that has no $\frac{\varepsilon}{R_n}$-fixed point?
%\end{question}

\bibliographystyle{plain}
\bibliography{discontinuousBrouwer}

\begin{thebibliography}{10}

\bibitem{AAF}
Micha{\l} Adamaszek, Henry Adams, and Florian Frick.
\newblock Metric reconstruction via optimal transport.
\newblock {\em SIAM Journal on Applied Algebra and Geometry}, 2(4):597--619,
  2018.

\bibitem{GH-BU-VR}
Henry Adams, Johnathan Bush, Nate Clause, Florian Frick, Mario G\'{o}mez,
  Michael Harrison, R.~Amzi Jeffs, Evgeniya Lagoda, Sunhyuk Lim, Facundo
  M\'{e}moli, Michael Moy, Nikola Sadovek, Matt Superdock, Daniel Vargas,
  Qingsong Wang, and Ling Zhou.
\newblock Gromov--{H}ausdorff distances, {B}orsuk--{U}lam theorems, and
  {V}ietoris--{R}ips complexes.
\newblock {\em arXiv preprint arXiv:2301.00246}, 2023.

\bibitem{adams2025quantifying}
Henry Adams, Florian Frick, Michael Harrison, Sunhyuk Lim, Nikola Sadovek, and
  Matt Superdock.
\newblock Quantifying discontinuity.
\newblock {\em arXiv preprint arXiv:2511.07636}, 2025.

\bibitem{akopyan2012borsuk}
Arseniy Akopyan, Roman Karasev, and Alexey Volovikov.
\newblock Borsuk--{U}lam type theorems for metric spaces.
\newblock {\em arXiv preprint arXiv:1209.1249}, 2012.

\bibitem{balitskiy2021bounds}
Alexey Balitskiy.
\newblock {\em Bounds on Urysohn width}.
\newblock PhD thesis, Massachusetts Institute of Technology, 2021.

\bibitem{bjorner2017using}
Anders Bj{\"o}rner, Ji{\v{r}}{\'\i} Matou{\v{s}}ek, and G{\"u}nter~M Ziegler.
\newblock Using {B}rouwer’s fixed point theorem.
\newblock {\em A Journey Through Discrete Mathematics: A Tribute to
  Ji{\v{r}}{\'\i} Matou{\v{s}}ek}, pages 221--271, 2017.

\bibitem{brouwer1911abbildung}
Luitzen Egbertus~Jan Brouwer.
\newblock {\"U}ber {A}bbildung von {M}annigfaltigkeiten.
\newblock {\em Mathematische Annalen}, 71(1):97--115, 1911.

\bibitem{dubins1981equidiscontinuity}
Lester Dubins and Gideon Schwarz.
\newblock Equidiscontinuity of {B}orsuk--{U}lam functions.
\newblock {\em Pacific Journal of Mathematics}, 95(1):51--59, 1981.

\bibitem{Jung1910}
Heinrich~W.E. Jung.
\newblock Über den kleinsten {K}reis, der eine ebene {F}igur einschließt.
\newblock {\em Journal für die reine und angewandte Mathematik}, 137:310--313,
  1910.

\bibitem{klee1961stability}
Victor Klee.
\newblock Stability of the fixed-point property.
\newblock In {\em Colloquium Mathematicum}, volume~1, pages 43--46, 1961.

\bibitem{lim2023gromov}
Sunhyuk Lim, Facundo M{\'e}moli, and Zane Smith.
\newblock The {G}romov--{H}ausdorff distance between spheres.
\newblock {\em Geometry \& Topology}, 27:3733--3800, 2023.

\bibitem{nash1950equilibrium}
John~F Nash~Jr.
\newblock Equilibrium points in $n$-person games.
\newblock {\em Proceedings of the National Academy of Sciences}, 36(1):48--49,
  1950.

\bibitem{308972}
Rozu.
\newblock {G}romov--{H}ausdorff distance between a disk and a circle.
\newblock MathOverflow.
\newblock \url{https://mathoverflow.net/q/308972}.

\bibitem{sitnikov1958rundheit}
Kirill Sitnikov.
\newblock {\"U}ber die {R}undheit der {K}ugel.
\newblock {\em Nachr.\ Akad.\ Wiss.\ G{\"o}ttingen, Math.-Physik Kl.\ IIa},
  9:213--215, 1958.

\bibitem{Vietoris27}
Leopold Vietoris.
\newblock {{\"U}ber den h{\"o}heren Zusammenhang kompakter R{\"a}ume und eine
  Klasse von zusammenhangstreuen Abbildungen}.
\newblock {\em Mathematische Annalen}, 97(1):454--472, 1927.

\end{thebibliography}

\end{document}